\documentclass[reqno,11pt]{amsart}
\usepackage{amsmath,amssymb,latexsym,soul,cite,mathrsfs}
\usepackage{color,enumitem,graphicx}
\usepackage[colorlinks=true,urlcolor=blue,
citecolor=red,linkcolor=blue,linktocpage,pdfpagelabels,
bookmarksnumbered,bookmarksopen]{hyperref}
\usepackage[english]{babel}
\usepackage[left=2.6cm,right=2.6cm,top=2.85cm,bottom=2.85cm]{geometry}
\usepackage[hyperpageref]{backref}
\usepackage{lipsum}
\usepackage{enumerate}

\usepackage{tikz}
\usetikzlibrary{shapes.geometric, arrows}

\newtheorem{theorem}{Theorem}[section]
\newtheorem{lemma}[theorem]{Lemma}

\newtheorem{corollary}[theorem]{Corollary}
\newtheorem{proposition}[theorem]{Proposition}
\newtheorem{remark}[theorem]{Remark}
\newtheorem{definition}[theorem]{Definition}

\newtheorem*{definitionA}{Definition A}

\newcommand{\innerthmname}{}% initialize

\theoremstyle{definition}

\makeatletter
\def\namedlabel#1#2{\begingroup
	#2%
	\def\@currentlabel{#2}%
	\phantomsection\label{#1}\endgroup
}
\makeatother

\newcommand{\dive}{\mathrm{div}}

\newcommand{\Sp}{\mathbb{S}}
\newcommand{\ep}{\varepsilon}

\newcommand{\ha}{\mathcal{H}}
\newcommand{\s}{\hspace{7pt}}

\newcommand{\R}{\mathbb{R}}
\newcommand{\N}{\mathbb{N}}

\newcommand{\vsp}{\vspace{2pt}}
\newcommand{\ns}{\hspace{-2.5pt}}

\title[Stable $s$-minimal cones in $\mathbb{R}^2$ are flat for $s \sim 0$]{Stable $s$-minimal cones in $\mathbb{R}^2$ are flat for $s \sim 0$}

\author[M. Caselli]{Michele Caselli}

\address[M. Caselli]{
	Scuola Normale Superiore
	\newline\indent 
	Piazza dei Cavalieri 7, 56126 Pisa, Italy}
\email{\href{mailto:michele.caselli@sns.it}{michele.caselli@sns.it}}

\makeatletter
\def\l@subsection{\@tocline{2}{0pt}{2.5pc}{5pc}{}}
\def\l@subsubsection{\@tocline{2}{0pt}{5pc}{7.5pc}{}}
\makeatother
\begin{document}
	
	\begin{abstract}

  For $s \in (0,1)$ small, we show that the only cones in $\mathbb{R}^2$ stationary for the $s$-perimeter and stable in $\mathbb{R}^2 \setminus \{0\}$ are half-planes. This is in direct contrast with the case of the classical perimeter or the regime $s$ close to $1$, where nontrivial cones as $\{xy>0\} \subset \mathbb{R}^2$ are stable for inner variations.

	\end{abstract}
	
	\maketitle

 \tableofcontents
	
	\section{Introduction}

 In this work, we prove that for $s$ close to zero, half-planes are the only $s$-minimal cones in $\R^2$ that are stable in $\R^2 \setminus \{0\}$. Here, by an $s$-minimal cone, we mean an open cone $E\subset \R^2$ that is an $s$-minimal surface (that is, a stationary set for the $s$-perimeter under inner variations). This result is purely nonlocal since it is in direct contrast with both the classical case (formally $s = 1$) and the regime where $s$ is close to $1$, where the cross $\textbf{X} = \{xy > 0\}$ is a nontrivial (i.e., not a half-plane) stationary cone in the plane that is stable for inner variations, in any reasonable sense. 

 \vsp 
 Nevertheless, for $s$ close to zero, the cross $\textbf{X}$ is unstable in $\R^2\setminus \{0\}$ (see \ref{sec: appendix}), and has infinite index by Corollary \ref{cor: fmi cones}. Our proof relies on the behavior of the best constant in Hardy's inequality for the $H^{\sigma}(\R)$ seminorm as $\sigma \downarrow 1/2$. 

\vsp 
Classification results for $s$-minimal cones in $\R^n$ have been previously proved in different ranges of $s$ and $n$ and with various hypotheses on the cone, such as minimality in compact subsets or stability. Before stating our main result precisely and some consequences in Subsection \ref{sbs: min-max}, let us recall the previous literature on this problem. Even though slightly different notions of stability have been used in the literature---see Subsection \ref{sbs: diff stability} for a complete discussion about this---the known classification results for $s$-minimal cones in $\R^n$ can be summarized as follows. 

\vsp 
The table below has to be read in this way: for a cone $E\subset \R^n$, the hypotheses in each row imply that the $E$ is a half-space. 

\vspace{5pt} 
\begin{center}
\setlength{\tabcolsep}{8pt} % Default value: 6pt
\renewcommand{\arraystretch}{1.4} % Default value: 1
\begin{tabular}{ |c|l|c| } 
\cline{2-3} \multicolumn{1}{c|}{}
 & \textbf{Class of cones:} & \textbf{Range of $s$:} \\ \hline
\cite{SV2min, SV1mon} & Minimizing in $\R^2$ & $ \forall \, s\in (0,1)$ \\  \hline 
\cite{CSV} & Stable by rearrangements\footnotemark in $\R^2$ & $ \forall s\in (0,1)$ \\  \hline 
\cite{CCScones} & Stable and smooth in $\R^3 \setminus \{0\}$ & $s$ close to $1$ \\  \hline
\cite{CDSV} & Stable and smooth in $\R^4 \setminus \{0\}$ & $s$ close to $1$ \\  \hline
\cite{CVRegLimitingArg} & Minimizing in $\R^n$, for $2\le n \le 7$& $s$ close to $1$ \\  \hline
\end{tabular}
\end{center}

\vspace{6pt} 

\footnotetext{The notion of stability by rearrangements, which is Definition A in Subsection \ref{sbs: diff stability}, is an ad-hoc notion of stability developed to get rid of cross-like singularities directly from the definition.}

Let us stress that the results in \cite{SV2min, SV1mon} provide, for all $s \in (0,1)$, the classification of cones in $\R^2$ minimizing the $s$-perimeter in compact sets, in accordance with the classical case. These results do not imply that stable $s$-minimal cones in $\R^2$ are flat. In fact, for the notion of stability that we consider in this work (Definition \ref{def: stability} below), which is the most natural one induced by inner variations and also used in similar contexts like stationary varifolds, this fact is not even believed to be true for all $s\in (0,1)$. Indeed, for inner variations, the cross $\textbf{X}$ is expected to be a stable $s$-minimal cone in $\R^2$ for $s$ close to $1$, again in accordance with the classical case. 

\vsp 
The following is the notion
of stability that we use is this work. 

\begin{definition}[Stability]\label{def: stability}
    Let $\mathcal{U} \subset \R^n $ be an open set and $E$ be an $s$-minimal surface in $\mathcal{U}$ (see Definition \ref{def: smin surf}). We say that $E$ is stable in $\mathcal{U}$ if for every bounded Lipschitz domain $ \Omega \Subset \mathcal{U} $ we have
\begin{equation*}
   \delta^2 {\rm Per}_s( E, \Omega)[X] :=  \frac{d^2}{d t^2} \bigg|_{t=0} {\rm Per}_s( \phi_t^X(E), \Omega) \ge  0 ,   \s\s \forall \, X\in C_c^2(\Omega; \R^n) , 
\end{equation*}    
where $\phi_t^X : \Omega \to \Omega $ is the flow of $X$ at time $t>0$.
\end{definition}

\begin{remark}\label{rem: well def second variation} By \cite[Lemma 3.11]{CFSfrac}, if ${\rm Per}_s(E, \Omega ) <+\infty$ and $X\in C_c^2(\Omega; \R^n)$ then the map $t\mapsto {\rm Per}_s( \phi_t^X(E), \Omega)$ is well-defined for all $t>0$ and of class $C^2$. Thus, the previous definition of stability is meaningful without any a priori assumption on the regularity of $ E$. 
\end{remark}

For $s$ close to zero, the situation could differ from that of the classical perimeter. For example, in \cite[Theorem 4]{LawsonCones}, for $s$ close to zero, the authors construct a non-flat $s$-minimal cone in $\R^7$ that is smooth and stable in $\R^7 \setminus \{0\}$. This is in contrast with the case of the classical perimeter since, by a celebrated result by Simons \cite{Simons}, for $3\le n \le 7$, the only cones in $\R^n$ that are smooth and stable in $\R^n\setminus \{0\}$ are the hyperplanes. We refer to \cite[Chapter 9]{Chern} for a simplified exposition of Simons' result and to \cite[Theorem 1.16]{CPcones} for a modern presentation. 

\vsp 
In this work, for small $s$, we prove the first classification result for stable $s$-minimal cones in $\R^2\setminus \{0\}$ in direct contrast with the case of the classical perimeter or the regime $s$ close to $1$. The precise statement of our main result is as follows. 

 \begin{theorem}\label{thm: main}
   There exists $s_\circ \in (0,1/2)$ with the following property. Let $s\in (0, s_\circ)$ and $E \subset \R^2$ be an $s$-minimal cone stable in $\R^2 \setminus \{0\}$ (see Definition \ref{def: stability}). Then $E$ is a half-plane.
\end{theorem}

Moreover, using the fact that $s$-minimal cones with finite Morse index outside the origin are stable outside the origin (which is a trivial observation in the classical case of the perimeter, but not entirely trivial for $s$-minimal cones, see the beginning of Section \ref{sec: morse}), we deduce that the conclusialon of Theorem \ref{thm: main} also holds for cones of finite Morse index. 

\begin{corollary}\label{cor: fmi cones}
   The classification of Theorem \ref{thm: main} holds for $s$-minimal cones of finite Morse index in $\R^2 \setminus \{0\}$ (see Definition \ref{def: Morse index}). 
\end{corollary}

Here and in the rest of this work, by finite Morse index, we mean with respect to the following notion introduced in \cite{CFSYau, Enric}. 

   \begin{definition}[Morse index]\label{def: Morse index}
       Let $\mathcal{U} \subset \R^n $ be an open set and $E$ be an $s$-minimal surface in $\mathcal{U}$ (see Definition \ref{def: smin surf}). We say that $ E$ has Morse index at most $m$ in $\mathcal{U}$ if for every bounded Lipschitz domain $ \Omega \Subset \mathcal{U} $, for every $(m+1)$ vector fields $X_1, \dotsc, X_{m+1}\in C^2_c(\Omega; \R^n)$, there exists coefficients $a_1, \dotsc , a_{m+1} \in \R $ such that $a_1^2+ \dotsc + a_{m+1}^2=1$ and 
       \begin{equation*}
           \delta^2 {\rm Per}_s( E, \Omega)[a_1X_1+\dotsc + a_{m+1} X_{m+1}] \ge 0.
       \end{equation*}
   \end{definition}
 
In the \ref{sec: appendix}, we give a direct proof that the cross $\textbf{X}$ is unstable in $\R^2\setminus \{0\}$, for $s$ close to $0$. Since $\textbf{X}$ is not flat, this fact is already contained in Theorem \ref{thm: main}. Nevertheless, we provide a very short independent proof since we believe it captures the main idea in the proof of Theorem \ref{thm: main}.

\subsection{Min-max curves and model singularities} \label{sbs: min-max}

On a closed Riemannian manifold $(M^n,g)$, the volume spectrum, introduced by Gromov, is a sequence of geometric invariants $\{\omega_p(M,g)\}_{p\in \N}$ called 
$p$-widths, which can be thought of as a nonlinear analog of the spectrum of the Laplacian. We refer to \cite[Section 2.2]{JLD} or \cite[Section 2]{CMsurface} for the precise definition of the $p$-widths. These $p$-widths play a crucial role in the theory of minimal hypersurfaces. In ambient dimension $3\le n \le 7$, each $p$-width equals the weighted area of the union of disjoint, connected, smooth, closed, embedded minimal hypersurfaces. These hypersurfaces can be chosen to satisfy a bound on their Morse index, meaning that the sum of the Morse indices of the connected components is at most $p$. 

\vsp 
For $n=2$, the situation is different as min-max methods on surfaces typically only produce stationary geodesic networks with regular support up to finitely many points (e.g., \cite{Pittsmmnets}), making the standard index control techniques ineffective. In this direction, Chodosh and Mantoulidis recently achieved a significant breakthrough in \cite{CMsurface}, showing that on surfaces, the $p$-widths are achieved by finite unions of closed immersed geodesics rather than simply geodesic nets.

\vsp 
In the case of geodesics on surfaces, the regularity of these objects cannot be improved, meaning that even for generic metrics, the min-max scheme will produce immersed geodesics that are not embedded. At every self-intersection point, the tangent cone to these geodesics consists of a finite union of distinct lines intersecting transversely, which cannot be ruled out. 

\vsp 
On the other hand, in this case of ambient dimension two, these multiple-junction model singularities (where the tangent cone consists of a finite union of distinct lines) are the only obstruction to a complete regularity of these geodesic nets arising from a min-max scheme. If one knew that there are no points of multiple-junction, then the net would be a finite union of disjoint, closed, embedded geodesics. This is even true in higher dimensions under suitable hypotheses. Indeed, by a deep result by Wickramasekera \cite{Wick}, for $3\le n \le 7$, any stationary, stable on its regular part codimension $1$ varifold without multiple-junctions is smooth (to say, it is supported on a finite union of disjoint,
smooth, embedded, connected hypersurfaces). 

\vsp 
Let us now turn to the implications of this work to ``fractional geodesics''. Recall that, similarly to the terminology used for sets of finite perimeter (e.g., \cite[Part II]{MaggiBook}), an $s$-minimal surface is, to be precise, a set $E \subset M$ with finite $s$-perimeter and zero first variation (Definition \ref{def: smin surf} below). Nevertheless, with a bit of abuse of the notion, we often refer to just its boundary $\partial E$ as ``the'' surface, which is a codimension one object.  

\vsp 
The main result of this work, together with the one by the author and collaborators in \cite{CFSYau}, implies that the situation for the fractional analog of the volume spectrum is, for $s$ small and $n=2$, drastically different from the classical one of geodesics. We refer to \cite{Enric, CFSYau} for the precise definition of the fractional widths $\{\ell_{s,p}(M,g)\}_{p\in \N}$ and for the proof that these are indeed attained by $s$-minimal surfaces $E^s_p$ with Morse index at most $p$ on $M$, in the sense of Definition \ref{def: Morse index} above. Moreover, by \cite[Proposition 3.27]{CFSYau}, these surfaces are slightly more than stationary for inner variations: they are viscosity solutions of the NMS (i.e., Nonlocal Minimal Surface) equation. 

\vsp 
Since having Morse index at most $p$ is a property that is stable under blow-up, by the monotonicity formula for $s$-minimal surfaces (see \cite[Lemma 6.2]{MSW} or \cite[Theorem 3.4]{CFSfrac}) and the BV-estimate in the finite Morse index case \cite[Theorem 5.4]{Enric} we have that, for every $x\in \partial E^s_p$, any blow-up of $E^s_p$ around $x$ is a cone in $\R^n$ of finite Morse index in $\R^n \setminus \{0\}$. Thus, if the ambient Riemannian manifold is two-dimensional and $s\in (0, s_\circ)$, where $s_\circ$ is the constant of Theorem \ref{thm: main}, every such cone is a half-plane. This fact, together with the improvement of flatness theorem for viscosity solutions of the NMS equation in \cite{CRS} (see also \cite[Theorem 4.14]{CFSYau}), implies that $\partial E^s_p$ has a unique flat tangent cone at every point. Then, one can deduce that $\partial E^s_p$ is smooth by arguing exactly as in the first part of the proof of \cite[Theorem 4.12]{CFSYau}. We refer the reader to \cite[Section 4]{CFSYau} for all the details of this blow-up procedure to show regularity. 

\vsp 
Hence, on a Riemannian surface $(M^2,g)$ and for $s$ small, the fractional widths $\{\ell_{s,p}(M^2,g)\}_{p\in \N}$ are attained by smooth $s$-minimal surfaces, that are a finite union of smooth embedded curves. 

\subsection{On the different notions of stability}\label{sbs: diff stability} 
In recent years, two different notions of stability for $s$-minimal surfaces have been used in the literature \cite{CSV, CCScones}. We refer to \cite[Section 2]{CCScones} for a discussion on these two notions of stability and why a second notion (other than the natural one coming from inner variations) was developed. For the sake of clarity, we recall their major difference here and the role played in this work. 

\vsp 
First, we have the natural notion of stability coming from inner variations of the set, which we call just \emph{stability} and is our Definition \ref{def: stability}. This is the usual notion of stability, which is also used in the context of stationary varifolds.

\vsp 
Secondly, we have a stronger notion of stability coming from \emph{``outer rearrangements''} of the set.

\begin{definitionA}[Definition 1.6 in \cite{CSV}]\label{def: defa} Let $\Omega \subset \R^n$ be a bounded open set and $E$ be a set with ${\rm Per}_s(E; \Omega)<+\infty$. Then, $E$ is said to be stable by rearrangements in $\Omega$ if, for every vector field $X\in C_c^2(\Omega; \R^n)$ there holds
    \begin{equation*}
    \liminf_{t\to 0^+} \frac{1}{t^2} \big( {\rm Per}_s( \phi_t^X(E) \cap E , \Omega) - {\rm Per}_s( E, \Omega) \big) \ge 0  , 
\end{equation*}  
and 
\begin{equation*}
    \liminf_{t\to 0^+} \frac{1}{t^2} \big( {\rm Per}_s( \phi_t^X(E) \cup E , \Omega) - {\rm Per}_s( E, \Omega) \big) \ge 0  , 
\end{equation*} 
where $\phi_t^X : \Omega \to \Omega$ is the flow of $X$ at time $t>0$.
\end{definitionA}

In \cite{CSV}, this notion is called just ``stability", but we believe this terminology to be slightly misleading for the reason that follows. While for sets with $C^2$ boundary in $\Omega$ and $s\in (0,1]$, Definition A is known to be equivalent to Definition \ref{def: stability} (see \cite[Remark 3.2]{CCScones} for a proof of this fact), for singular objects they do not coincide in general. In particular, stability by rearrangements (Definition A) allows to infinitesimally break the topology of the set $E$, while classical stability by inner variations does not. We emphasize that the notion of stability by rearrangements was developed ad-hoc to get rid of cross-like singularities directly from the definition. 

\begin{remark}
    In the case of the classical perimeter (formally $s=1$), these two notions of stability are indeed different, since the cross $\textbf{X}$ is stable in $\R^2$ but is not stable by rearrangements in $\R^2$. Nevertheless, since $\textbf{X}$ is smooth outside the origin, $\textbf{X}$ is stable for both notions in $\R^2 \setminus \{0\}$. 
\end{remark}

Observe that our Definition \ref{def: stability} is weaker than Definition A. Indeed, assume that $E$ is stable by rearrangements in $\Omega$. For every $X\in C_c^2(\Omega; \R^n)$, by the elementary inequality
\begin{equation*}
    {\rm Per}_s( \phi_t^X(E) , \Omega) + {\rm Per}_s( E, \Omega) \ge {\rm Per}_s( \phi_t^X(E) \cap E, \Omega) + {\rm Per}_s( \phi_t^X(E) \cup E , \Omega) ,  
\end{equation*}
it follows that
\begin{equation}\label{eq: liminf Def A}
     \liminf_{t\to 0^+} \frac{1}{t^2} \big( {\rm Per}_s( \phi_t^X(E) , \Omega) - {\rm Per}_s( E, \Omega) \big) \ge 0 . 
\end{equation}
By \cite[Lemma 3.11]{CFSfrac}, the map $t\mapsto {\rm Per}_s( \phi_t^X(E), \Omega)$ is of class $C^2$ for all $t>0$. Hence, the limits in
    \begin{equation*}
        a_1 := \frac{d}{d t} \bigg|_{t=0} {\rm Per}_s( \phi_t^X(E), \Omega) \s \mbox{and} \s a_2 := \frac{d^2}{d t^2} \bigg|_{t=0} {\rm Per}_s( \phi_t^X(E), \Omega)
    \end{equation*}
    exist. It follows that $a_1=0$ (assuming the opposite, for some $X \in C_c^2(\Omega; \R^n)$ one would get the $\liminf$ in \eqref{eq: liminf Def A} to be $-\infty$), and consequently 
    \begin{equation*}
       \frac{a_2}{2} = \lim_{t\to 0^+} \frac{1}{t^2} \big( {\rm Per}_s( \phi_t^X(E) , \Omega) - {\rm Per}_s( E, \Omega) - t a_1  \big) = \lim_{t\to 0^+} \frac{1}{t^2} \big( {\rm Per}_s( \phi_t^X(E) , \Omega) - {\rm Per}_s( E, \Omega) \big)  \ge 0 .
    \end{equation*}
    Thus, $E$ is stable in $\Omega$ for Definition \ref{def: stability}.

\begin{remark} To be precise, in \cite{CCScones} the authors take \eqref{eq: liminf Def A} as their definition of stability. Even though this seems to be slightly different from the notion of stability used in this work (Definition \ref{def: stability}), since $t\mapsto {\rm Per}_s( \phi_t^X(E), \Omega)$ is of class $C^2$ for $X\in C_c^2(\Omega; \R^n)$, by the argument above the two notions are equivalent without any additional hypotheses on the set $E$. 
\end{remark}

To explain the surprising part of this work further, let us start by recalling the precise statement of the two available results in the literature on the classification of $s$-minimal cones in $\R^2$. 

\begin{theorem}[\hspace{.001pt}\cite{SV2min, SV1mon}]\label{thm: min cones classification} Let $s\in (0,1)$ and $E\subset \R^2$ be an $s$-minimal cone minimizing the $s$-perimeter in compact sets of $\R^2$. Then, $E$ is a half-plane. 
\end{theorem}

\begin{theorem}[Corollary 1.16 in \cite{CSV}]\label{thm: csv class} Let $s\in (0,1)$ and $E\subset \R^2$ be an $s$-minimal cone that is stable by rearrangements in $\R^2$ (i.e., for Definition A for every $\Omega \Subset \R^2$). Then, $E$ is a half-plane. 
\end{theorem}

Since every set that minimizes the $s$-perimeter in compact sets is stable by rearrangements, we see that Theorem \ref{thm: csv class} is stronger than the previous Theorem \ref{thm: min cones classification}. The key point for which our result is, in the range $s$ small, more surprising than Theorem \ref{thm: csv class} is the following: Theorem \ref{thm: csv class} assumes stability (by rearrangement) also around the vertex of the cone, while our Theorem \ref{thm: main} assumes stability just in $\R^2 \setminus \{0\}$. This is a major difference since Theorem \ref{thm: csv class} does not hold only assuming that $E$ is stable by rearrangements in $\R^2 \setminus \{0\}$. Actually, this classification of cones in the plane stable in $\R^2 \setminus \{0\}$ is not even expected to be true for all $s\in (0,1)$, since the cross $\textbf{X}$ is expected to be stable in $\R^2 \setminus \{0\}$ for the $s$-perimeter and $s$ close to $1$, in accordance with the case of the classical perimeter. 

\vsp 
In this regard, our classification result Theorem \ref{thm: main} for cones in the plane stable in $\R^2\setminus \{0\}$ is of purely nonlocal nature and represents a remarkable difference from the theory of the classical perimeter.

\section{Preliminary tools} 

Our proof relies on the Hardy inequality for the $H^\sigma(\mathbb{R})$ seminorm and, specifically, on the asymptotic behavior of its optimal constant as $\sigma \downarrow 1/2$. The sharp constant in this inequality has been established in \cite{HardyIneq} (equation (1.6)), and it is also stated in \cite[Theorem 3.3]{CCScones}. We will also use the fact that radially symmetric functions in $C^2_c(\mathbb{R}\setminus \{0\})$ nearly saturate Hardy's inequality; this is proved in Section 3.3 of \cite{HardyIneq}.

 \begin{theorem}[Hardy's inequality] Let $n\ge 1$, $\sigma \in (0, 1 )$ and $u\in H^\sigma_0(\R^n \setminus \{0\})$. Then 

\begin{equation*}
    \ha _{n, \sigma} \int_{\R^n} \frac{u^2}{|x|^{2\sigma}} \, dx \le c_{n,\sigma} \iint_{\R^n \times \R^n} \frac{|u(x)-u(y)|^2}{|x-y|^{n+2\sigma}} \, dxdy , 
\end{equation*}
where 
\begin{equation*}
    \ha _{n,\sigma} = 2^{2\sigma-1}(\sigma-n/2)^2 \frac{\Gamma(n/4+\sigma/2)^2}{\Gamma(n/4-\sigma/2+1)^2} ,
\end{equation*}
and 
\begin{equation*}
   c_{n,\sigma} = 2^{2\sigma-1} \pi^{-n/2} \frac{\Gamma(n/2+\sigma)}{\Gamma(2-\sigma)} \sigma (1-\sigma) .
\end{equation*}

Moreover, the inequality is saturated by radial functions, that is: for every $\ep>0$ there exists a radial function $\xi(x) = \xi(|x|) \in C^2_c(\R^n \setminus \{0\})$ such that
\begin{equation*}
    \big( \ha _{n,\sigma} + \ep \big) \int_{\R^n} \frac{\xi(x)^2}{|x|^{2\sigma}} \, dx \ge c_{n,\sigma} \iint_{\R^n \times \R^n} \frac{|\xi(x)-\xi(y)|^2}{|x-y|^{n+2\sigma}} \, dxdy .
\end{equation*}
     
 \end{theorem}
 
 In particular, for $n=1$, $s\in (0,1/2)$ and $\sigma = \frac{1+s}{2} $ the constant reads as
     \begin{equation*}
    \ha _{1,\tfrac{1+s}{2}} =  s^2 2^{s-2} \frac{\Gamma(\frac{2+s}{4})^2}{\Gamma(\frac{4-s}{4})^2} ,
\end{equation*}
and, by elementary properties of the Gamma function, it can be easily checked that
\begin{equation*}
   \frac{\ha _{1,\tfrac{1+s}{2}}}{c_{1,\tfrac{1+s}{2}}} \le \frac{Cs^2}{C\tfrac{1+s}{2}(1-\tfrac{1+s}{2})} \le Cs^2 , 
\end{equation*}
for some absolute $C>0$ and every $s\in (0,1/2)$. 

\vsp 
Summarizing, taking the $\xi$ relative to $\ep= \ha _{1,\tfrac{1+s}{2}} $ in the saturation statement above, for every $s\in (0,1/2)$ there exists an even function $\xi \in C^2_c(\R\setminus \{0\})$ such that
\begin{equation*}
    \int_\R \frac{\xi(x)^2}{|x|^{1+s}} \, dx \ge  \frac{1}{Cs^2 } \iint_{\R \times \R}  \frac{|\xi(x)-\xi(y)|^2}{|x-y|^{2+s}} \, dxdy .
\end{equation*}
Lastly, for the same $\xi$, this directly implies
\begin{equation}\label{eq: Hardy sat}
    \int_0^\infty \frac{\xi(x)^2}{x^{1+s}} \, dx \ge  \frac{1}{Cs^2 } \int_0^\infty \hspace{-4pt} \int_0^\infty   \frac{|\xi(x)-\xi(y)|^2}{|x-y|^{2+s}} \, dxdy .
\end{equation}

Before stating precisely all the tools that we will need on the first and second variation of the fractional perimeter, we recall the notion of fractional $s$-perimeter, which was introduced by Caﬀarelli, Roquejoﬀre, and Savin in \cite{CRS}.

 \begin{definition}\label{def: frac per}
    For $s\in (0,1)$, the fractional perimeter (or $s$-perimeter) of a measurable set $E\subset \R^n$ is defined as
    \begin{equation*}
         {\rm Per}_s(E) := \frac{1}{2} [\chi_E ]^2_{H^{s/2}(\R^n)} =  \iint_{E\times E^c} \frac{1}{|x-y|^{n+s}} \,dxdy .
    \end{equation*}
\end{definition}

The $s$-perimeter also has a natural localized version in a bounded open set $\Omega \subset \R^n$, in the same spirit of the localized fractional Sobolev spaces $H^s(\Omega)$. This is of use because, for example, one would like to say that a hyperplane in $\R^n$ is an $s$-minimal surface (see Definition \ref{def: smin surf} below) even though a half-space has infinite $s$-perimeter for Definition \ref{def: frac per}. 

\begin{definition}
    Let $s\in (0,1)$ and $\Omega \subset \R^n$ be a bounded, open set. The fractional perimeter (or $s$-perimeter) of a measurable set $E $ in $\Omega$ is defined as
    \begin{align*}
         {\rm Per}_s(E , \Omega) : = \frac{1}{2} \iint_{\R^n \times \R^n \setminus \Omega^c \times \Omega^c} \frac{|\chi_E(x)-\chi_E(y)|^2}{|x-y|^{n+s}} \,dxdy .
    \end{align*}
\end{definition}

Note that for $\Omega = \R^n$ we recover Definition \ref{def: frac per}.

\begin{definition}[$s$-minimal surface]\label{def: smin surf} Given $\mathcal{U} \subset \R^n$ open, a set $E$ is said to be an $s$-minimal surface in $\mathcal{U}$ if for every bounded Lipschitz domain $ \Omega \Subset \mathcal{U} $ we have ${\rm Per}_s( E, \Omega )<+\infty$ and
\begin{equation*}
    \frac{d}{d t} \bigg|_{t=0} {\rm Per}_s( \phi_t^X(E), \Omega ) =0 , \s\s \forall \, X\in C_c^2(\Omega ; \R^n) , 
\end{equation*}    
where $\phi_t^X : \Omega \to \Omega $ is the flow of $X$ at time $t>0$.
\end{definition}

By the first variation formula (e.g, \cite[Theorem 6.1]{FFMMM} or \cite[Proposition 5.1]{Enric}), if $E$ is an $s$-minimal surface in $\R^n$ and $\partial E$ is of class $C^2$ then 
\begin{equation}\label{eq: fvar}
     P.V. \ns \int_{\R^n} \frac{\chi_{E^c}(y)-\chi_E(y)}{|x-y|^{n+s}} \, dy = 0 , \s\s \forall \, x\in \partial E . 
\end{equation}
The left-hand side is denoted by $H_s^E(x)$ and called the \emph{nonlocal mean curvature} of $E$ at $x$.

\begin{remark}\label{rem: nbh first var} Inspecting the proof of the first variation formula in \cite{FFMMM, Enric}, it follows that if $\partial E$ is $C^2$ just in a neighborhood of $x\in \partial E$ and is globally Lipschitz, then \eqref{eq: fvar} holds at $x$. 
\end{remark}

It is well known that, for sets with $C^2$ boundary, the nonlocal mean curvature can be expressed as a boundary integral (see, for example, the introduction of \cite{CabreFall} or \cite{CirFig}). Precisely 
 \begin{equation}\label{eq: boundary int formula}
        H_s^E(x) = \frac{2}{s} P.V. \ns  \int_{\partial E} \frac{(y-x) \cdot \nu_{\partial E}(y)}{|y-x|^{n+s}} \, d\sigma(y) ,
    \end{equation}
    where $\nu$ denotes the outer unit normal to $\partial E$. 

    \vsp 
In the proof of our main theorem, we will need to use this fact for a cone in $\R^2$, which is not necessarily $C^2$ globally. The validity of this representation formula \eqref{eq: boundary int formula} at some $x\in \partial E$ does not really require the smoothness of $\partial E$ everywhere, but just smoothness in a neighborhood of $x$ and mild regularity of $\partial E$ outside this neighborhood. Since the author could not find any precise statement of this fact, and since we will crucially need it, we state and prove it here. 

\begin{lemma}\label{lem: boundary int formula} Let $s\in (0,1)$, and let $E\subset \R^n$ be a set that is a finite union of open domains with Lipschitz boundary. Let $x\in \partial E$ and assume that $\partial E$ is $C^2$ in a neighborhood of $x$. Then $H_s^E(x)$ can be expressed as in \eqref{eq: boundary int formula}. 
\end{lemma}

\begin{proof}
    Fix $\ep>0$ small. As the statement is invariant by translations, we can assume $x=0 \in \partial E$. Observe that 
    \begin{equation*}
     \dive \left( \frac{y}{|y|^{n+s}} \right) = -  \frac{s}{|y|^{n+s}} . 
\end{equation*}
We want to apply the divergence theorem to the vector field $F \in C^1(\R^n \setminus B_\ep; \R^n) $ defined by $F(y)= y|y|^{-(n+s)}$, but this has not compact support. Nevertheless, since $|F| = o(|y|^{1-n})$ as $|y| \to + \infty$, it is a classical fact that the divergence theorem can be applied (e.g., \cite[Corollary 5.2]{Langbook}). Hence, by the divergence theorem, we can infer
\begin{equation*}
    \int_{\partial(E\setminus B_\ep)} \frac{y \cdot \nu(y)}{|y|^{n+s}} \, d\sigma(y) =  \int_{E\setminus B_\ep} \dive \left( \frac{y}{|y|^{n+s}} \right)  dy = - s \int_{E\setminus B_\ep} \frac{1}{|y|^{n+s}} \, dy , 
\end{equation*}
where $\nu$ is the outer unit normal to $\partial(E\setminus B_\ep)$. Similarly 
\begin{equation*}
    \int_{\partial(E^c \setminus B_\ep)} \frac{y \cdot \widetilde \nu(y)}{|y|^{n+s}} \, d\sigma(y) = - s \int_{E^c\setminus B_\ep} \frac{1}{|y|^{n+s}} \, dy , 
\end{equation*}
where $\widetilde \nu$ is the outer unit normal to $\partial(E^c\setminus B_\ep)$. Note that 
\begin{equation*}
    \partial(E\setminus B_\ep) = ( \partial E\setminus B_\ep ) \cup (\partial B_\ep \cap E^c ) , \s \mbox{and } \partial(E^c\setminus B_\ep) = ( \partial E^c\setminus B_\ep ) \cup ( \partial B_\ep \cap E ) , 
\end{equation*}
and that $\widetilde \nu = -\nu $ on the shared part of the boundary $\partial E\setminus B_\ep = \partial E^c\setminus B_\ep$. Thus, subtracting the two equalities above 
\begin{align*}
    s \int_{\R^n \setminus B_\ep} \frac{\chi_{E^c}(y)-\chi_E(y)}{|y|^{n+s}} \, dy &  = \int_{\partial(E\setminus B_\ep)} \frac{y \cdot \nu(y)}{|y|^{n+s}} \, d\sigma(y) - \int_{\partial(E^c \setminus B_\ep)} \frac{y \cdot \widetilde \nu(y)}{|y|^{n+s}} \, d\sigma(y) \\ &= 2 \int_{\partial E \setminus B_\ep} \frac{y \cdot \nu(y)}{|y|^{n+s}} \, d\sigma(y) - \int_{\partial B_\ep \cap E^c} \frac{d\sigma(y)}{|y|^{n+s-1}} + \int_{\partial B_\ep \cap E} \frac{d\sigma(y)}{|y|^{n+s-1}} . 
\end{align*}

Moreover 
\begin{equation*}
    - \int_{\partial B_\ep \cap E^c} \frac{d\sigma(y)}{|y|^{n+s-1}} + \int_{\partial B_\ep \cap E} \frac{d\sigma(y)}{|y|^{n+s-1}} = \frac{1}{\ep^{n+s-1}} \big( \mathcal{H}^{n-1}(\partial B_\ep \cap E) -  \mathcal{H}^{n-1}(\partial B_\ep \cap E^c) \big) , 
\end{equation*}
and it easily follows from the hypothesis that $\partial E$ is $C^2$ in a neighborhood of $0$ that 
\begin{equation*}
    \mathcal{H}^{n-1}(\partial B_\ep \cap E) -  \mathcal{H}^{n-1}(\partial B_\ep \cap E^c) = O(\ep^{n}).
\end{equation*}
Thus 
\begin{equation*}
    - \int_{\partial B_\ep \cap E^c} \frac{d\sigma(y)}{|y|^{n+s-1}} + \int_{\partial B_\ep \cap E} \frac{d\sigma(y)}{|y|^{n+s-1}} = O(\ep^{1-s}) \to 0 , \s \mbox{as } \ep \to 0^+. 
\end{equation*}
Letting $\ep \to 0^+$ above and diving both sides by $s$ gives
\begin{equation*}
    P.V. \ns \int_{\R^n} \frac{\chi_{E^c}(y)-\chi_E(y)}{|y|^{n+s}} \, dy = \frac{2}{s} P.V.  \ns \int_{\partial E } \frac{y \cdot \nu_{\partial E}(y)}{|y|^{n+s}} \, d\sigma(y) , 
\end{equation*}
which is what we wanted to prove. 
\end{proof}

If $\partial E$ is smooth, the second variation can be written in a form very reminiscent of the second variation formula for classical minimal surfaces. This second variation formula for smooth $s$-minimal surfaces was proved in \cite{LawsonCones, FFMMM} , and has recently been generalized to ambient Riemannian manifolds in \cite{Enric}.

\begin{theorem}\label{thm: second var} Let $ E$ be an $s$-minimal surface in $\R^n$, and assume that $\partial E$ is $C^2$ in some bounded open set $\Omega$. Then, for every $X \in C^2_c(\partial E \cap \Omega ; \R^n)$, setting $\varphi := X\cdot \nu_{\partial E}$, we have
     \begin{equation*}
     \delta^2 {\rm Per}_s( E ; \Omega )[X] =   \iint_{\partial E \times \partial E } \left(  \frac{|\varphi(x)-\varphi(y)|^2}{|x-y|^{n+s}} - \frac{|\nu_{\partial E}(x)- \nu_{\partial E}(y)|^2}{|x-y|^{n+s}} \varphi(x)^2 \right) d\sigma(x) d\sigma(y) , 
\end{equation*}
where $\nu_{\partial E}$ is the outer unit normal to $\partial E$. In particular, if $E$ is stable in $\Omega$, there holds
\begin{equation}\label{eq: stab}
     \iint_{\partial E \times \partial E} \frac{|\nu_{\partial E}(x)- \nu_{\partial E}(y)|^2}{|x-y|^{n+s}} \varphi(x)^2 \, d\sigma(x) d\sigma(y)  \le \iint_{\partial E \times \partial E } \frac{|\varphi(x)-\varphi(y)|^2}{|x-y|^{n+s}} \, d\sigma(x) d\sigma(y) ,
\end{equation}
for every $\varphi \in C^2_c(\Omega)$. 
    
\end{theorem}

\subsection{The BV-estimate for small \texorpdfstring{$s$}{}}

It is known \cite{CSV, Enric} that stable $s$-minimal surfaces enjoy a uniform interior BV-estimate, and that the same holds for stable solutions of the fractional Allen-Cahn equation \cite{CCSAC, CFSYau}. The results in these references only control the dependence of the constant from $s$ as $s \to 1$. In this work, we need the same type of BV-estimate with a control of the constant as $s\to 0$. 

\begin{theorem}\label{thm: BV s close to zero}
    Let $R>0$, $s\in (0,1/2)$, and $E\subset \R^n$ be an $s$-minimal surface which stable in $B_R(x)$ and such that $\partial E$ is $C^2$ in $B_R(x)$. Then 
    \begin{equation*}
        {\rm Per}(E, B_{R/2}(x)) \le \frac{C}{s} ,
    \end{equation*}
    for some dimensional constant $C>0$.
\end{theorem}

We first recall a useful interpolation inequality for the $s$-perimeter that accounts for the dependence on $s$ both for $s \to 1$ and $s \to 0$. The inequality in the form we use here is not explicitly stated anywhere; it is written throughout the lines of the proof of Theorem 3.1 in \cite{Jack24}, or it follows from \cite[Lemma 3.13]{Enric} and Young's inequality. 

\begin{lemma}\label{lem: interpolation}
    Let $s\in (0,1)$, $R>0$, and $E$ be a set with locally finite perimeter. Then
    \begin{equation*}
        {\rm Per}_s(E, B_R(x)) \le  C \left( \frac{ R^{1-s}}{1-s}  {\rm Per}(E, B_{5R}(x)) +   \frac{R^{n-s}}{s} \right) ,
    \end{equation*}
    for some dimensional constant $C>0$. 
\end{lemma}

With this inequality, we can deduce the BV-estimate for small $s$.

\begin{proof}[Proof of Theorem \ref{thm: BV s close to zero}]
   Similarly to \cite{CCScones}, the theorem follows by inspection of the proof of Theorem 1.7 in \cite{CSV}, taking care of the explicit dependence of the constants as $s\to 0$. For the sake of clarity, we rewrite here the crucial estimates in the proof of Theorem 1.7 in \cite{CSV}, with the precise dependence of all constants on $s$, as $s\to 0 $. In the proof that follows $C>0$ is a dimensional constant that can change from line to line. 

Since the statement is scaling and translation invariant, we can assume $R=1$ and $x=0$. Since $E$ is stable in $B_1$, by Theorem 1.9 in \cite{CSV} applied to the kernel $K(z)=1/|z|^{n+s}$, we get 
\begin{equation}\label{eq: qwfqwd}
    {\rm Per}(E, B_1) \le C \left(  1+ \sqrt{{\rm Per}_s(E, B_4)} \right) .
\end{equation}
Moreover, by Lemma \ref{lem: interpolation} applied with $R=4$ we have
\begin{equation*}
    {\rm Per}_s(E, B_4) \le C\left( \frac{1}{1-s} {\rm Per}(E, B_{20}) +  \frac{1}{s} \right) \le  C\left( {\rm Per}(E, B_{20}) +  \frac{1}{s} \right), 
\end{equation*}
where we have also used that $s\le 1/2$. Thus, by \eqref{eq: qwfqwd} and Young's inequality we get
\begin{align*}
    {\rm Per}(E, B_1) \le C\left( 1 +  \sqrt{{\rm Per}(E, B_{20}) + \frac{1}{s}} \right)  & \le C\left( 1 + \delta {\rm Per}(E, B_{20}) + \frac{\delta}{s} + \frac{1}{\delta} \right)  \\ &= C \left( 1 +  \frac{\delta}{s} + \frac{1}{\delta} \right) +   \delta {\rm Per}(E, B_{20}) , 
\end{align*}
for every $\delta >0$. From here, arguing exactly as the end of the proof of \cite[Theorem 1.7]{CSV} or \cite[Proposition 3.14]{CFSYau}, choosing $\delta$ smaller than a dimensional constant $\delta_\circ=\delta_{\circ}(n) >0$ and a covering argument one concludes the uniform bound
\begin{equation*}
    {\rm Per}(E, B_{1/2}) \le C \left( 1 +  \frac{\delta_\circ }{s} + \frac{1}{\delta_\circ} \right) \le \frac{C}{s} . 
\end{equation*}
\end{proof}

\section{Classification of stable cones for small \texorpdfstring{$s$}{}}

Let us fix the notation that we will use for cones in the plane. For a cone $E \subset \R^2$, we write $\Sigma = \partial E$ and observe that $\Sigma$ is a union of half-lines from the origin. Write
\begin{equation*}
    E = \bigcup_{i=1}^N E_i \,, \s \Sigma  = \bigcup_{i=1}^{2N} \Sigma_i \,, \s \partial E_i = \{\Sigma_i, \Sigma_{i+1}\} , 
\end{equation*}
with the convention that $\Sigma_{2N+1}= \Sigma_1$. 
Here $E_i$ are disjoint conical sectors from the origin, that is $\lambda E_i= E_i$ for every $\lambda>0$, $\Sigma_i$ are rays from the origin with the induced orientation from $E_i$, and the number $N$ could be $+\infty$ in general, but will be finite in our proof. 

We also denote by $\theta_{i}^{\hspace{.5pt}j}$ the counterclockwise angle from $\Sigma_i$ and $\Sigma_j$.

\begin{lemma}\label{lem: rays ker lower bound} In the notation above, there is $c>0$ such that for every $x \in \Sigma_j$ there holds
    \begin{equation*}
            \int_{ \Sigma_i }   \frac{1}{ |x-y |^{2+s}} \, d\sigma(y) \ge \frac{c}{|x|^{1+s} (1-\cos(\theta_{i}^{\hspace{.5pt}j}))^{1+s}} .
    \end{equation*}
\end{lemma}

\begin{proof} 
    We have 
\begin{align*}
     \int_{ \Sigma_i }   \frac{1}{|x-y|^{2+s}} \, d\sigma(y) & =  \int_{\Sigma_i } \frac{d\sigma(y)}{(|x|^2 + |y|^2-2\cos(\theta_{i}^{\hspace{.5pt}j}) |x| |y| )^{\frac{2+s}{2}}}   \\  &=  \int_0^\infty \frac{dz}{(|x|^2+z^2-2\cos(\theta_{i}^{\hspace{.5pt}j})|x|z )^{\frac{2+s}{2}}}    =  \frac{1}{|x|^{1+s}} \int_0^\infty \frac{dt}{(1+t^2-2t\cos(\theta_{i}^{\hspace{.5pt}j}) )^{\frac{2+s}{2}}}  ,
     \end{align*}
where we have substituted $z=t|x|$ in the last line. Moreover
\begingroup
\allowdisplaybreaks
\begin{align*}
   \int_0^\infty & \frac{dt}{(1+t^2-2t\cos(\theta_{i}^{\hspace{.5pt}j}) )^{\frac{2+s}{2}}} =  \int_0^\infty \frac{dt}{((t-1)^2+2t(1-\cos(\theta_{i}^{\hspace{.5pt}j})) )^{\frac{2+s}{2}}} \\ & \ge \int_{1/2}^{3/2} \frac{dt}{((t-1)^2 + 2t(1-\cos(\theta_{i}^{\hspace{.5pt}j})) )^{\frac{2+s}{2}}} \ge  \int_{1/2}^{3/2} \frac{dt}{((t-1)^2 + 3(1-\cos(\theta_{i}^{\hspace{.5pt}j})) )^{\frac{2+s}{2}}}  \\ &= \int_{-1/2}^{1/2} \frac{dt}{((t-1)^2 + 3(1-\cos(\theta_{i}^{\hspace{.5pt}j})) )^{\frac{2+s}{2}}} \ge  \int_{1/2}^{3/2} \frac{dt}{(t^2 + 3(1-\cos(\theta_{i}^{\hspace{.5pt}j})) )^{\frac{2+s}{2}}}  \\ & \ge \frac{1}{(1-\cos(\theta_{i}^{\hspace{.5pt}j}))^{1+s}}    \int_{-1/10}^{1/10} \frac{dt}{(t^2 + 3 )^{\frac{2+s}{2}}} \ge \frac{c}{(1-\cos(\theta_{i}^{\hspace{.5pt}j}))^{1+s}} .
\end{align*}
\endgroup
This concludes the proof.
\end{proof}

Now, we have all the ingredients to prove our main result Theorem \ref{thm: main}. In the proof, we plug in the stability inequality a radial test function that nearly saturates Hardy's inequality on $(0, \infty)$, in the sense that \eqref{eq: Hardy sat} holds.

\begin{proof}[Proof of Theorem \ref{thm: main}]
    By Theorem \ref{thm: BV s close to zero}, that is the BV-estimate for stable $s$-minimal surfaces for $s \in (0,1/2)$, and a standard covering argument we get that 
    \begin{equation*}
       2N = {\rm Per}(E, B_2 \setminus B_{1}) \le \frac{C}{s} .
    \end{equation*}
    Hence $\Sigma = \partial E$ is a finite number of rays from the origin, whose number is bounded by
    \begin{equation}\label{eq: number bound above}
        2N\le \frac{C}{s} . 
    \end{equation}

Recall the stability inequality \eqref{eq: stab}, and let $\nu_i$ be the outer unit normal to $\Sigma_i$ from $E_i$. For the left hand side, for every $\varphi \in C^2_c(\R^2 \setminus \{0\})$, we have

\begin{align*}
     \iint_{\Sigma \times \Sigma} \frac{|\nu_\Sigma(x)- \nu_\Sigma(y)|^2}{|x-y|^{2+s}} \varphi(x)^2 & \, d\sigma(x) d\sigma(y)  = \sum_{i,j} \iint_{\Sigma_i \times \Sigma_j} \frac{|\nu_i(x)- \nu_j(y)|^2}{|x-y|^{2+s}} \varphi(x)^2 \, d\sigma(x) d\sigma(y) \\ &= 2 \sum_{i\neq j} (1-(-1)^{i+j}\cos(\theta_{i}^{\hspace{.5pt}j})) \iint_{\Sigma_i \times \Sigma_j} \frac{\varphi(x)^2}{|x-y|^{2+s}} \, d\sigma(x) d\sigma(y) . 
\end{align*}
By Lemma \ref{lem: rays ker lower bound} we can estimate
\begin{align*}
    \iint_{ \Sigma_i    \times \Sigma_j}   \frac{\varphi(x)^2}{|x-y |^{2+s}} \, d\sigma(x)d\sigma(y) \ge \frac{c}{(1-\cos(\theta_{i}^{\hspace{.5pt}j}))^{1+s}} \int_{\Sigma_j } \frac{\varphi(x)^2}{|x|^{1+s}} \, dx . 
\end{align*}
Now, taking $\varphi(x)=\xi(|x|)$ with $\xi$ saturating Hardy's inequality on $(0, \infty)$ as in \eqref{eq: Hardy sat}, gives
\begin{equation*}
    \int_{\Sigma_j} \frac{\xi(x)^2}{|x|^{1+s}} \, d\sigma(x) \ge \frac{1}{Cs^2 } \iint_{\Sigma_j \times \Sigma_j}  \frac{|\xi(x)-\xi(y)|^2}{|x-y|^{2+s}} \, d\sigma(x)d\sigma(y) ,
\end{equation*}
thus
\begin{align}
    \iint_{\Sigma \times \Sigma} & \frac{|\nu_\Sigma(x)- \nu_\Sigma(y)|^2}{|x-y|^{1+s}} \xi(|x|)^2 \, d\sigma(x) d\sigma(y) \nonumber \\ &  \ge \frac{c}{s^2} \sum_{j=1}^{2N} \left( \iint_{\Sigma_j \times \Sigma_j}  \frac{|\xi(x)-\xi(y)|^2}{|x-y|^{2+s}} \, d\sigma(x)d\sigma(y) \right) \sum_{1 \le i \le 2N, \, i \neq j} \frac{1-(-1)^{i+j}\cos(\theta_{i}^{\hspace{.5pt}j})}{(1-\cos(\theta_{i}^{\hspace{.5pt}j}))^{1+s}} . \nonumber \\ &  = \frac{c}{s^2} \left( \int_0^\infty \hspace{-4pt} \int_0^\infty  \frac{|\xi(x)-\xi(y)|^2}{|x-y|^{2+s}} \, dxdy \right) \sum_{j=1}^{2N} \sum_{ \{1 \le i \le 2N, \, i \neq j \} } \frac{1-(-1)^{i+j}\cos(\theta_{i}^{\hspace{.5pt}j})}{(1-\cos(\theta_{i}^{\hspace{.5pt}j}))^{1+s}} . \label{eq: for claim}
\end{align}

\textbf{Claim.} There exists $s_\circ <1/2$ sufficiently small with the following property. For every $s\in (0, s_\circ)$, there exists $j \in \{1, 2, \dotsc , 2N\}$ such that 
\begin{equation}\label{eq: small interaction}
    \sum_{ \{1 \le i \le 2N, \, i \neq j \}} \frac{1-(-1)^{i+j}\cos(\theta_{i}^{\hspace{.5pt}j})}{(1-\cos(\theta_{i}^{\hspace{.5pt}j}))^{1+s}} \le \frac{1}{100}. 
\end{equation}
Indeed, suppose this is not the case. Then, for $s$ arbitrary small, \eqref{eq: for claim} implies

\begin{equation*}
    \iint_{\Sigma \times \Sigma} \frac{|\nu_\Sigma(x)- \nu_\Sigma(y)|^2}{|x-y|^{1+s}} \xi(|x|)^2 \, d\sigma(x) d\sigma(y) \ge \frac{c N }{100 s^2}  \int_0^\infty \hspace{-4pt} \int_0^\infty  \frac{|\xi(x)-\xi(y)|^2}{|x-y|^{2+s}} \, dxdy . 
\end{equation*}
From this inequality, using that $E$ is stable gives
\begin{align*}
    \frac{c N }{100 s^2}   \int_0^\infty \hspace{-4pt} \int_0^\infty &  \frac{|\xi(x)-\xi(y)|^2}{|x-y|^{2+s}} \, dxdy \\ &\le  \iint_{\Sigma \times \Sigma} \frac{|\nu_\Sigma(x)- \nu_\Sigma(y)|^2}{|x-y|^{1+s}} \xi(|x|)^2 \, d\sigma(x) d\sigma(y)  \\ & \stackrel{{\rm (stability)}}{\le}  \iint_{\Sigma \times \Sigma } \frac{|\xi(|x|)-\xi(|y|)|^2}{|x-y|^{2+s}} \, d\sigma(x) d\sigma(y) \\ & =  N \int_0^\infty \hspace{-4pt} \int_0^\infty  \frac{|\xi(x)-\xi(y)|^2}{|x-y|^{2+s}} \, dxdy + \sum_{i \neq j} \iint_{\Sigma_i \times \Sigma_j}  \frac{| \xi(|x|)- \xi(|y|)|^2}{|x-y|^{2+s}} \, d\sigma(x)d\sigma(y) . 
\end{align*}
Moreover, since $|x-y| \ge ||x|-|y|| $, we have 
\begin{align*}
     \iint_{\Sigma_i \times \Sigma_j}  \frac{|\xi(|x|)-\xi(|y|)|^2}{|x-y|^{2+s}} \, d\sigma(x)d\sigma(y) & \le   \iint_{\Sigma_i \times \Sigma_j}  \frac{|\xi(|x|)-\xi(|y|)|^2}{||x|-|y||^{2+s}} \, d\sigma(x)d\sigma(y) \\ &= \int_0^\infty \hspace{-4pt} \int_0^\infty  \frac{|\xi(x)-\xi(y)|^2}{|x-y|^{2+s}} \, dxdy , 
\end{align*}
for every $i \neq j$. Thus 
\begin{equation*}
     \frac{c N }{100 s^2}   \int_0^\infty \hspace{-4pt} \int_0^\infty  \frac{|\xi(x)-\xi(y)|^2}{|x-y|^{2+s}} \, dxdy \le N^2 \int_0^\infty \hspace{-4pt} \int_0^\infty  \frac{|\xi(x)-\xi(y)|^2}{|x-y|^{2+s}} \, dxdy , 
\end{equation*}
which implies, together with \eqref{eq: number bound above}, that
\begin{equation*}
   s^2 \ge \frac{c}{100N} \ge \frac{c}{50C} s , 
\end{equation*}
which gives a contradiction if $s$ is small. Hence, the claim is proved. 

\vsp 
Now, we conclude the proof of our theorem by contradiction, with $s_\circ$ the one given by the claim above. Assume by contradiction that $N\ge 2$. Then, for the index $j$ such that \eqref{eq: small interaction} holds, we get that (here, as above, the indices are modulo $2N$)
\begin{equation*} 
    \frac{1 - (-1)^{j+(j+2)} \cos(\theta_{j}^{j+2})}{(1 - \cos(\theta_{j}^{j+2}))^{1+s}} = \frac{1}{(1 - \cos(\theta_{j}^{j+2}))^{s}}  \le \frac{1}{100} ,
\end{equation*}
holds for every $s\le s_\circ$. Clearly, this is not possible for any value of $\theta_{j}^{j+2} \in [0, 2\pi)$, and hence we conclude that $N=1$ and $E$ is made only of one conical sector of angle $\theta$. 

At this point, we would formally like to argue that, being $E$ an $s$-minimal surface, the first variation formula \eqref{eq: fvar} at $x=0$ implies that $E$ is a half-space. Nevertheless, the first-variation formula holds only at points $x\in \partial E$ where $\partial E$ is $C^2$ in a neighborhood of $x$ (see Remark \ref{rem: nbh first var}), and this is not the case for the tip of the cone. Hence, we have to argue in this spirit but a bit more indirectly.

With no loss of generality, up to rotation and complementation, assume that $\theta \in (0, \pi] $ and 
\begin{equation*}
    E = \{ \rho e^{i\varphi} \, : \, \rho >0, \varphi \in (0,\theta) \}. 
\end{equation*}
Let $(\textbf{t}_1, \textbf{t}_2) = \textbf{t} := e^{i\theta} \in \Sp^1$, so that $\partial E = \Sigma= \Sigma_0 \cup \Sigma_1$ where $\Sigma_0 = \{ \lambda e_1 \, : \, \lambda>0\}$ and $\Sigma_1 = \{\lambda \textbf{t} \, : \, \lambda>0\}$. Observe that, since $\theta \in (0,\pi]$, we have that $\textbf{t}_2\ge 0$. 

Since $e_1 = (1,0) \in \Sigma $ and $\Sigma$ is smooth in a neighborhood of $e_1$ (see Remark \ref{rem: nbh first var}), by the first-variation formula 
\begin{equation*}
    H_s^E(e_1) = P.V. \ns \int_{\R^2}  \frac{\chi_{E^c}(y)-\chi_E(y)}{|y - e_1 |^{2+s}} \, dy =0. 
\end{equation*}
Moreover, by Lemma \ref{lem: boundary int formula} this implies 
\begin{equation*}
   P.V. \ns  \int_{\Sigma} \frac{(y-e_1) \cdot \nu_\Sigma(y)}{|y-e_1|^{2+s}} \, d\sigma(y) =0 . 
\end{equation*}

 Note that the respective outer-unit normals to $\Sigma_0$ and $\Sigma_1$ are $-e_2$ and $\textbf{t}^\perp := i \textbf{t}$, respectively. Hence $(y-e_1) \cdot \nu_\Sigma(y) =0$ for all $y \in \Sigma_0$, and $(y-e_1) \cdot \nu_\Sigma(y) = -\textbf{t}^\perp_1 = \textbf{t}_2 \ge 0$ for all $y \in \Sigma_1$. In particular $(y-e_1) \cdot \nu_\Sigma(y)  \ge 0$ for all $y \in \Sigma$. Thus
\begin{equation*}
    0= P.V. \ns  \int_{\Sigma} \frac{(y-e_1) \cdot \nu_\Sigma(y)}{|y-e_1|^{2+s}} \, d\sigma(y) \ge 0 . 
\end{equation*}
This means that all the inequalities above are equalities, that is
\begin{equation*}
     \textbf{t}_2 = 0 \implies \textbf{t} = \pm e_1 \implies \theta \in \{0, \pi\} . 
\end{equation*}
Since $\theta \in (0,\pi]$, necessarily $\theta=\pi$ and thus $E$ is a half-space.
\end{proof}

\section{Extension to cones of finite index}\label{sec: morse}
   
In this section, we show that our classification for $s$-minimal cones stable in $\R^2\setminus \{0\}$ implies the classification of cones with finite Morse index in $\R^2\setminus \{0\}$. However, to establish this implication, we must show that every regular cone $E \subset \R^n$ stationary for the $s$-perimeter and with finite Morse index in $\R^n \setminus \{0\}$ is stable in $\R^n\setminus \{0\}$.

    \begin{proposition}\label{prop: fmi implies stable outside}
    Let $s\in (0,1)$ and $E\subset \R^n$ be an $s$-minimal cone with $C^2$ boundary in $\R^n \setminus \{0\}$ and with finite Morse index in $\R^{n}\setminus \{0\}$ (see Definition \ref{def: Morse index}). Then $E$ is stable in $\R^n \setminus \{0\}$. 
\end{proposition}
    
    In the classical case of the perimeter (formally $s=1$), this property follows easily by a scaling argument; if $E$ were unstable in $\R^n\setminus \{0\}$, one could construct infinitely many disjoint scaled copies of an unstable variation, contradicting the finite index assumption. The fractional setting presents additional difficulty since the nonlocal interactions between different scales prevent such a direct argument, as functions with disjoint support are not orthogonal for the $H^s$ scalar product. 

\vsp 
This kind of result has been previously established in \cite{CFSYau} for blow-ups of $s$-minimal surfaces arising as limits of the fractional Allen-Cahn equation. Our proof follows a similar strategy and is essentially already contained in the union of \cite{CFSYau} and \cite{Enric}. Nevertheless, since the result as needed in this work is never stated explicitly nor proved, we provide a proof in this section. 

\begin{lemma}\label{lem: weigthed as}
    Let $\mathcal{U} \subset \R^n$ be an open set and $E$ be an $s$-minimal surface in $\mathcal{U}$ with Morse index at most $m$ in $\mathcal{U}$ (see Definition \ref{def: Morse index}). Assume also that $\Sigma := \partial E $ is $C^2$ in $\mathcal{U}$. For every bounded Lipschitz domain $\Omega \Subset \mathcal{U}$ let $X_1,X_2,...,X_{m+1} \in C_c^2(\Omega; \R^n)$ be vector fields with disjoint compact supports $A_1,...,A_{m+1} \subset \Omega $, and denote $D_{k\ell}:=\textnormal{dist}(A_k,A_\ell)$.
For $1\leq i<\ell \leq m+1$, fix positive weights $\lambda_{i\ell}>0$. Then, for at least one of the $i$ (depending on $E$) we have that
\begin{equation*}
\delta^2 {\rm Per}_s(E; A_i)[X_i] \geq - C \|X_i\|_{L^\infty}^2\textnormal{diam}(A_i)^{2(n-1)}\left(\sum_{\ell<i} \frac{1}{\lambda_{\ell i}} D_{i \ell}^{-(n+s)}
    +\sum_{\ell>i} \lambda_{i \ell} D_{i\ell}^{-(n+s)}\right) ,
\end{equation*}
for some constant $C=C(n,s,m)>0$.
\end{lemma}

\begin{proof} 
The statement is a more precise version of Lemma 5.8 in \cite{Enric}, and the proof proceeds similarly. Using the second variation formula, which can be used since $\Sigma$ is $C^2$ in $\mathcal{U}$, we compute the second variation of $E$ for linear combinations of $m+1$ vector fields $X_i$ , supported each in the corresponding $A_i$. We denote by $\xi_i := X_i \cdot \nu_\Sigma$ the scalar normal component of these vector fields. 

By the second variation formula of Theorem \ref{thm: second var} we have  
\begin{equation*}
\begin{split}
    \delta^2 {\rm Per}_s(E; \Omega) &[a_1 X_1+a_2 X_2+\dotsc+a_{m+1} X_{m+1}]   \\ \vspace{6pt}
    & = a_1^2\delta^2 {\rm Per}_s(E; A_1)[X_1] + \dotsc+a_{m+1}^2\delta^2 {\rm Per}_s(E; A_{m+1})[X_{m+1}] \\
    & \hspace{0.2cm} +2 a_1a_2 \iint_{ (\Sigma \cap A_1) \times (\Sigma \cap A_2)} \frac{(\xi_1(x) - \xi_1(x))(\xi_2(y) - \xi_2(y))}{|x-y|^{n+s}} d\sigma(x)d\sigma(y)\\
    & \hspace{0.2cm}+\dotsc\\
    & \hspace{0.2cm}+2 a_m a_{m+1} \iint_{(\Sigma \cap A_m) \times (\Sigma\cap A_{m+1})} \frac{(\xi_m(x) - \xi_m(x))(\xi_{m+1}(y) - \xi_{m+1}(y))}{|x-y|^{n+s}} d\sigma(x)d\sigma(y).
\end{split}
\end{equation*}
Recall that the supports of $\xi_i$ and $\xi_j$ are $A_i$ and $A_j$ respectively, which are disjoint. Then, the term containing the double integral over $A_i \times A_j$ with $i<j$ can be bounded as
\begin{equation*}
\begin{split}
   2 a_ia_j  \iint_{(\Sigma \cap A_i) \times (\Sigma \cap A_j)} & \frac{(\xi_i(x) - \xi_i(x))(\xi_j(y) - \xi_j(y))}{|x-y|^{n+s}} d\sigma(x)d\sigma(y) \\ &= -2a_ia_j\iint_{(\Sigma \cap A_i) \times (\Sigma \cap A_j)} \frac{\xi_i(x) \xi_j(y)}{|x-y|^{n+s}} d\sigma(x)d\sigma(y) \\
    &\leq 2|a_ia_j| D_{ij}^{-(n+s)}\|\xi_i\|_{L^1(\Sigma \cap A_i)}\|\xi_j\|_{L^1(\Sigma \cap A_j)} \\
     &\leq \lambda_{ij} a_i^2  D_{ij}^{-(n+s)}\|\xi_i\|_{L^1(\Sigma \cap A_i)}^2+\frac{1}{\lambda_{i j}} a_j^2D_{i j}^{-(n+s)}\|\xi_j\|_{L^1(\Sigma \cap A_j)}^2 ,
\end{split}
\end{equation*}
where we have used Young's inequality in the last line. Substituting this into the second variation expression above gives
\begin{align*}
    \delta^2 {\rm Per}_s(E; \Omega) & [a_1 X_1+a_2 X_2+\dotsc+a_{m+1} X_{m+1}]   \\  \le
    &  \sum_{i=1}^{m+1} a_i^2\left[\delta^2 {\rm Per}_s(E; A_i)[X_i] +\|\xi_i\|_{L^1(\Sigma \cap A_i)}^2 \left(\sum_{j<i} \frac{1}{\lambda_{ji}} D_{ij}^{-(n+s)}
    +\sum_{j>i} \lambda_{ij} D_{ij}^{-(n+s)} \right)\right].
\end{align*}

The condition that the Morse index is at most $m$ implies that the expression cannot be $<0$ for all $(a_1,\dotsc,a_{m+1})\neq 0$. Hence, we find that there must exist some $i$ such that 
\begin{equation*}
\begin{split}
    \delta^2 {\rm Per}_s(E; A_i)[X_i] \geq - \|\xi_i\|_{L^1(\Sigma \cap A_i)}^2 \left(\sum_{j<i} \frac{1}{\lambda_{ji}} D_{ij}^{-(n+s)}
    +\sum_{j>i} \lambda_{ij} D_{ij}^{-(n+s)} \right)
\end{split}
\end{equation*}
holds for all $\xi_i \in C_c^1(\Sigma \cap A_i)$. Moreover, for $z\in A_i $, we have
\begin{align*}
    \|\xi_i\|_{L^1(\Sigma \cap A_i)}^2 & = \left( \int_{A_i \cap \Sigma} X_i \cdot \nu_\Sigma \, d\sigma  \right)^2 \le \| X_i\|^2_{L^\infty} \mathcal{H}^{n-1}(A_i \cap \Sigma)^2 \\ & \le  \| X_i\|^2_{L^\infty} {\rm Per}(E, B_{{\rm diam}(A_i)}(z))^2 \le C \| X_i\|^2_{L^\infty}    {\rm diam}(A_i)^{2(n-1)} . 
\end{align*}
Putting everything together concludes the proof. 
\end{proof}

Recall also the following result, which is \cite[Lemma 4.15]{CFSYau}.

\begin{lemma}\label{lem: stable cones CFS} Let $s\in (0,1)$ and $E \subset \R^n$ be a cone with ${\rm Per}_s(E, B_1(0)) <+\infty $. Assume that $E$ is an $s$-minimal surface and that satisfies the property in Lemma \ref{lem: weigthed as} with $\mathcal{U} = \R^n \setminus \{0\}$. Then $E$ is stable in $\R^n\setminus \{0\}$.
\end{lemma}

With this, we can easily deduce Proposition \ref{prop: fmi implies stable outside}. 

\begin{proof}[Proof of Proposition \ref{prop: fmi implies stable outside}] 
Since $E$ has finite Morse index in $\R^n \setminus \{0\}$, it satisfies the property of Lemma \ref{lem: weigthed as} with $\mathcal{U} = \R^n \setminus \{0\}$. Moreover, by \cite[Theorem 5.4]{Enric} and a simple covering argument we have that 
\begin{equation*}
    {\rm Per}_s(E, B_1(0)) <+\infty . 
\end{equation*}
From here, the result follows by Lemma \ref{lem: stable cones CFS}.  
\end{proof}

\section{Appendix}\makeatletter\def\@currentlabel{Appendix}\makeatother\label{sec: appendix}

Here, we give a direct proof that $\textbf{X} = \{xy>0 \} \subset \R^2$ is unstable for variations compactly supported in $\R^2 \setminus \{0\}$. First, note that by symmetry, it is clear that $\textbf{X}$ satisfies \eqref{eq: fvar} at every $x\in \partial \textbf{X}$ and for every $s\in (0,1)$. Hence, since $\textbf{X}$ is smooth away from the origin, $\textbf{X}$ is an $s$-minimal surface for every $s\in (0,1)$.  

\begin{proposition}
 There exists $s_\circ \in (0,1/2) $ such that, for $s\in (0,s_\circ)$, $\textbf{X}$ is unstable in $\R^2 \setminus \{0\}$.
\end{proposition}

\begin{proof} 
Let $L_x := \{(x,0) \in \R^2 \, : \, x>0\} $ and $L_y := \{(0,y) \in \R^2\, : \, y>0\} $. Choose a radial test function $\varphi = \varphi(|x|) \in C^2_c(\R^2\setminus \{0\})$. With a little abuse of notation, we still denote by $\varphi$ its trace on the lines in $\partial \textbf{X}$. Let also $\nu$ be the outer unit normal to $\partial \textbf{X}$. 

On the one hand 
\begin{align*}
     \iint_{\partial \textbf{X} \times \partial \textbf{X}} & \frac{|\nu(x)- \nu(y)|^2}{|x-y|^{2+s}}  \varphi(|x|)^2  \, d\sigma(x) d\sigma(y) \\ & = 4 \int_0^\infty \hspace{-4pt} \int_{-\infty}^0 \frac{2^2}{|x-y|^{2+s}} \varphi(x)^2  dxdy + 8 \int_{ L_x } \int_{ L_y } \frac{(\sqrt{2})^2}{|x-y|^{2+s}}  \varphi(|x|)^2 d\sigma(x)d\sigma(y) \\ & = 16 \int_0^\infty \hspace{-4pt} \int_{-\infty}^0 \frac{\varphi(x)^2}{|x-y|^{2+s}} \,  dxdy + 16 \int_{ L_x } \int_{ L_y } \frac{\varphi(|x|)^2}{|x-y|^{2+s}} \, d\sigma(x)d\sigma(y) .
    \end{align*}
Note that 
\begin{align*}
    \int_{ L_x }  \int_{ L_y }  \frac{\varphi(|x|)^2}{|x-y|^{2+s}} \, d\sigma(x)d\sigma(y) & = \int_{ L_x } \varphi(|x|)^2 \, d\sigma(x) \int_{L_y } \frac{d\sigma(y)}{(x^2+y^2)^{\frac{2+s}{2}}} \,  \\ &= \int_0^\infty \varphi(x)^2 \left( \frac{1}{x^{1+s}}\int_0^\infty \frac{dt}{(1+t^2)^{\frac{2+s}{2}}}\right)   dx \\ &= \frac{\sqrt{\pi}\, \Gamma \big( \tfrac{1+s}{2} \big) }{2 \Gamma \big( \tfrac{2+s}{2} \big) } \int_0^\infty \frac{\varphi(x)^2}{x^{1+s}} \, dx , 
\end{align*}
and similarly 
\begin{equation*}
    \int_0^\infty \hspace{-4pt} \int_{-\infty}^0 \frac{\varphi(x)^2}{|x-y|^{2+s}} \,  dxdy = \frac{1}{1+s} \int_0^\infty \frac{\varphi(x)^2}{x^{1+s}} \, dx  . 
\end{equation*}
Hence 
    \begin{align}\label{eq 1}
        \iint_{\partial \textbf{X} \times \partial \textbf{X}} \frac{|\nu(x)- \nu(y)|^2}{|x-y|^{2+s}} \varphi^2(x) \, d\sigma(x) d\sigma(y) = 16 \left( \frac{1}{1+s} + \frac{\sqrt{\pi}\, \Gamma \big( \tfrac{1+s}{2} \big) }{2 \Gamma \big( \tfrac{2+s}{2} \big) } \right)  \int_0^\infty \frac{\varphi(x)^2}{x^{1+s}} \, dx . 
    \end{align}
    
On the other hand, for the Sobolev part in the stability inequality
\begin{align}\label{eq 2}
    \iint_{\partial \textbf{X} \times \partial \textbf{X} } & \frac{|\varphi(|x|)-\varphi(|y|)|^2}{|x-y|^{2+s}} \, d\sigma(x) d\sigma(y) \le 100 \int_0^\infty \hspace{-4pt} \int_0^\infty \frac{ |\varphi(x)-\varphi(y)|^2}{|x-y|^{2+s}} \, dxdy .
\end{align}

Now, we use the fact that Hardy's inequality is saturated. Choose $\varphi(x)=\xi(|x|)$ where $\xi$ saturates the Hardy's inequality as in \eqref{eq: Hardy sat}. Applying \eqref{eq 1} and \eqref{eq 2} with $\varphi(x)=\xi(|x|)$ we obtain
\begin{align*}
     \iint_{\partial \textbf{X} \times \partial \textbf{X}} \frac{|\nu(x)- \nu(y)|^2}{|x-y|^{2+s}} &  \xi^2(x) \, d\sigma(x) d\sigma(y) = 16 \left( \frac{1}{1+s} + \frac{\sqrt{\pi}\, \Gamma \big( \tfrac{1+s}{2} \big) }{2 \Gamma \big( \tfrac{2+s}{2} \big) } \right) \int_{0}^\infty \frac{\xi(x)^2}{x^{1+s}} \, dx \\ &  \stackrel{\eqref{eq: Hardy sat}}{\ge} \frac{16}{Cs^2} \left( \frac{1}{1+s} + \frac{\sqrt{\pi}\, \Gamma \big( \tfrac{1+s}{2} \big) }{2 \Gamma \big( \tfrac{2+s}{2} \big) } \right) \int_0^\infty \hspace{-4pt} \int_0^\infty  \frac{|\xi(x)-\xi(y)|^2}{|x-y|^{2+s}} \, dxdy \\ & \ge  \frac{16}{100Cs^2} \left( \frac{1}{1+s} + \frac{\sqrt{\pi}\, \Gamma \big( \tfrac{1+s}{2} \big) }{2 \Gamma \big( \tfrac{2+s}{2} \big) } \right)  \iint_{\partial \textbf{X} \times \partial \textbf{X} } \frac{|\xi(x)-\xi(y)|^2}{|x-y|^{2+s}} \, d\sigma(x) d\sigma(y) .
\end{align*}
Thus, in order to contradict stability, it is sufficient that 
\begin{equation*}
   \frac{16}{100Cs^2} \left( \frac{1}{1+s} + \frac{\sqrt{\pi}\, \Gamma \big( \tfrac{1+s}{2} \big) }{2 \Gamma \big( \tfrac{2+s}{2} \big) } \right) \ge 1 , 
\end{equation*}
which is clearly the case if $s$ is sufficiently small. Hence, $\textbf{X}$ is unstable in this range.
\end{proof}

\vspace{.5cm}
\noindent 
\textbf{Acknowledgements.} The author is extremely grateful to Enric Florit-Simon and Mattia Freguglia for their many valuable comments on a preliminary version of this work.

 %%%%%%%%%%%%%%%%%%%%%%%%%%%%%%%%%%%%%%%%%%%%%%%%%%%%%%%%%%%%%%%%%%%%%%%%%%%%%%%%%%%%%%%%%%%%%%%%%%%
	% BIBLIOGRAPHY %%%%%%%%%%%%%%%%%%%%%%%%%%%%%%%%%%%%%%%%%%%%%%%%%%%%%%%%%%%%%%%%%%%%%%%%%%%%%%%%%%%%
	%%%%%%%%%%%%%%%%%%%%%%%%%%%%%%%%%%%%%%%%%%%%%%%%%%%%%%%%%%%%%%%%%%%%%%%%%%%%%%%%%%%%%%%%%%%%%%%%%%%
	 	 	\bibliography{references}
	 	 	\bibliographystyle{alpha}
            % \bibliographystyle{abbrv}
    % \printbibliography
	 	 	
\end{document}